\documentclass[final,1p,times]{elsarticle}

\usepackage[colorlinks=true]{hyperref}
\usepackage{mathrsfs}
\usepackage{amsthm,amsmath,amssymb,mathrsfs,amsfonts,graphicx,graphics,latexsym,exscale,cmmib57,dsfont,bbm,amscd,ulem}
\usepackage{color,soul}
\newtheorem{thm}{Theorem}[section]
\newtheorem{lemma}[thm]{Lemma}
\newtheorem{cor}[thm]{Corollary}
\newtheorem*{corB}{Corollary}
\newtheorem*{cor1}{Corollary~1}
\newtheorem*{cor2}{Corollary~2}
\theoremstyle{definition}
\newtheorem{defn}[thm]{Definition}
\newtheorem{conj}[thm]{Conjecture}
\newtheorem{prop}[thm]{Proposition}

\numberwithin{equation}{section}
\def\x{\boldsymbol{x}}
\def\P{\boldsymbol{P}}
\journal{arXiv, July 2, 2015}

\begin{document}

\begin{frontmatter}
\title{On chaotic minimal center of attraction of a Lagrange stable motion for topological semi flows}%

\author{Xiongping Dai}
\ead{xpdai@nju.edu.cn}

\address{Department of Mathematics, Nanjing University, Nanjing 210093, People's Republic of China}

\begin{abstract}
Let $f\colon\mathbb{R}_+\times X\rightarrow X$ be a topological semi flow on a Polish space $X$. In 1977, Karl Sigmund conjectured that if there is a point $\x$ in $X$ such that the motion $f(t,\x)$ has just $X$ as its minimal center of attraction, then the set of all such $\x$ is \textit{residual} in $X$. In this paper, we present a positive solution to this conjecture and apply it to the study of chaotic dynamics of minimal center of attraction of a motion.
\end{abstract}

\begin{keyword}
Minimal center of attraction $\cdot$ Chaotic motion $\cdot$ Semi flow

\medskip
\MSC[2010] 37B05 $\cdot$ 37A60 $\cdot$ 54H20
\end{keyword}
\end{frontmatter}


\section{Introduction}\label{sec1}

By a \textit{$C^0$-semi flow} over a metric space $X$, we here mean a transformation
$f\colon \mathbb{R}_+\times X\rightarrow X$ where $\mathbb{R}_+=[0,\infty)$, which satisfies the following three conditions:
\begin{enumerate}
\item[(1)] The initial condition: $f(0,x)=x$ for all $x\in X$.
\item[(2)] The condition of continuity: if there be given two convergent sequences $t_n\to t_0$ in $\mathbb{R}_+$ and $x_n\to x_0$ in $X$, then $f(t_n,x_n)\to f(t_0,x_0)$ as $n\to\infty$.
\item[(3)] The semigroup condition: $f(t_2,f(t_1,x))=f(t_1+t_2,x)$ for any $x$ in $X$ and any times $t_1,t_2$ in $\mathbb{R}_+$.
\end{enumerate}
Sometimes we write $f(t,x)=f^t(x)$ for any $t\ge0$ and $x\in X$; and for any given point $x\in X$ we call $f(t,x)$ a motion in $X$ and $\mathcal{O}_f(x)=f(\mathbb{R}_+,x)$ the orbit starting from the point $x$. If $\mathcal{O}_f(x)$ is precompact (i.e. $\overline{\mathcal{O}_f(x)}$ is compact) in $X$, then we say that $f(t,x)$ is \textit{Lagrange stable}.

We refer to any subset $\Lambda$ of $X$ as an \textit{invariant} set if $f(t,x)\in\Lambda$ for each pint $x\in\Lambda$ and any time $t\ge0$. In dynamical systems, statistical mechanics and ergodic theory, we shall have to do with
``probability of sojourn of a motion $f(t,x)$ in a given region $E$ of $X$" as
$t\to+\infty$:
\begin{equation*}
\P(f(t,x)\in
E)=\lim_{T\to+\infty}\frac{1}{T}\int_0^T\mathds{1}_{E}^{}(f(t,x))dt,
\end{equation*}
where $\mathds{1}_{E}^{}(x)$ is the characteristic function of the set
$E$ on $X$.

This motivates H.F.~Hilmy to introduce following important concept, which was discussed in \cite{J,NS,Sig77,HZ}, for example.

\begin{defn}[Hilmy 1936~\cite{Hilmy}]\label{def1.1}
Given any $x\in X$, a closed subset $C_x$ of $X$ is called the {\it center of
attraction of the motion $f(t,x)$} as $t\to+\infty$ if
$\P(f(t,x)\in B_\varepsilon(C_x))=1$ for all $\varepsilon>0$, where $B_\varepsilon(C_x)$ denotes the
$\varepsilon$-neighborhood around $C_x$ in $X$. If the set $C_x$ does
not admit a proper subset which is likewise a center of attraction of the motion $f(t,x)$ as $t\to+\infty$,
then $C_x$ is called the {\it minimal center of attraction of the
motion $f(t,x)$} as $t\to+\infty$.
\end{defn}

First of all, by the classical Cantor-Baire theorem and Zorn's lemma we can obtain the following basic existence lemma.

\begin{lemma}\label{lem1.2}
Let $f\colon \mathbb{R}\times X\rightarrow X$ be a $C^0$-semi flow on a metric space $X$. Then each Lagrange stable motion $f(t,x)$ always possesses the minimal center of  attraction.
\end{lemma}

From now on, by $\mathcal{C}_x$ we will understand the minimal center of attraction of a Lagrange stable motion $f(t,x)$ as $t\to+\infty$.
In \cite{Sig77}, Karl~Sigmund gave an intrinsic characterization for $\mathcal{C}_x$.
In this paper, we shall study the generic property and chaotic dynamics occurring in $\mathcal{C}_x$ for a Lagrange stable motion $f(t,x)$.


Just as the existence of one point which is topologically transitive implies that a residual set is topologically transitive, in 1977 Karl Sigmund raised the following open problem:

\begin{conj}[{K.~Sigmund 1977~\cite[Remark~4]{Sig77}}]\label{conj1.3}
For any homeomorphism $f$ of a compact metric space $X$, the set of points $x\in X$ with $\mathcal{C}_x=X$, if nonempty, is residual in $X$.
\end{conj}

Since a residual set contains a dense and $G_\delta$ subset of $X$, it is very large from the viewpoint of topology. Although Sigmund's conjecture is of interest, there has not been any progresses on it since 1977 except $f$ satisfies the ``specification'' property (\cite[Proposition~6]{Sig77}).
In Section~\ref{sec2}, we will present a positive solution to Sigmund's conjecture without any imposed shadowing assumption like specification, which is stated as follows:

\begin{thm}\label{thm1.4}
Let $f\colon \mathbb{R}_+\times X\rightarrow X$ be a $C^0$-semi flow on a Polish space $X$. If some motion $f(t,\x)$ is such that $\mathcal{C}_{\x}=X$, then the set $\{x\in X\colon \mathcal{C}_x=X\}$ is residual in $X$.
\end{thm}

Our argument of Theorem~\ref{thm1.4} below also works for discrete-time case. Thus we can obtain the following:

\begin{corB}
For any continuous transformation $f$ of a Polish space $X$, the set of points $x\in X$ with $\mathcal{C}_x=X$, if nonempty, is residual in $X$.
\end{corB}

We now turn to some applications of Theorem~\ref{thm1.4}. For our convenience, we first introduce two notions for a $C^0$-semi flow $f\colon \mathbb{R}_+\times X\rightarrow X$ on a Polish space $X$.

\begin{defn}\label{def1.5}
An $f$-invariant subset $\Lambda$ of $X$ is referred to as \textit{generic} if there exists some point $x\in\Lambda$ with $\Lambda=\mathcal{C}_x$.
\end{defn}

According to Conjecture~\ref{conj1.3} (or precisely speaking Theorem~\ref{thm1.4}) for any generic minimal center of attraction $\mathcal{C}_x$ of a motion $f(t,x)$, the set of points $y\in\mathcal{C}_x$ with $\mathcal{C}_x=\mathcal{C}_y$ is residual in $\mathcal{C}_x$. 

\begin{defn}\label{def1.6}
We say that a motion $f(t,x)$ is \textit{chaotic} for $f$ if there can be found some point $y\in X$ such that 
\begin{gather*}
\liminf_{t\to+\infty}d(f(t,x),y)=0,\quad\limsup_{t\to+\infty}d(f(t,x),y)>0
\intertext{and}\liminf_{t\to+\infty}d(f(t,x),f(t,y))=0,\quad\limsup_{t\to+\infty}d(f(t,x),f(t,y))>0.
\end{gather*}
\end{defn}

By using Theorem~\ref{thm1.4}, we will show that if $\mathcal{C}_x$ is not generic, then the chaotic behavior occurs near $\mathcal{C}_x$; see Theorems~\ref{thm3.1} and \ref{thm3.2} stated and proved in Section~\ref{sec3}. On the other hand whenever $\mathcal{C}_x$ is generic and it is not ``very simple'', then chaotic motions are generic in $\mathcal{C}_x$; that is the following

\begin{thm}\label{thm1.7}
Let $f(t,p)$ be a Lagrange stable motion in a Polish space $X$. If $\mathcal{C}_p$ is generic and itself is not a minimal subset of $(X,f)$, then there exists a residual subset $S$ of $\mathcal{C}_p$ such that $f(t,x)$ is chaotic for each $x\in S$.
\end{thm}

In Section~\ref{sec4} we will consider a relationship of the minimal center of attraction of a motion with the pointwise recurrence.

Finally we shall consider the multiply attracting of the minimal center of attraction of a motion in Section~\ref{sec5}.
\section{Proof of Sigmund's conjecture}\label{sec2}
This section will be devoted to proving Karl Sigmund's Conjecture~\ref{conj1.3} in the continuous-time case; that is, Theorem~\ref{thm1.4}.

Recall that if a continuous surjective map $T\colon X\rightarrow X$ is topologically transitive, then for a countable basis $U_1,U_2,\dotsc,U_n,\dotsc$ of the underlying space $X$, the set
\begin{equation*}
\left\{x\in X\,\big{|}\,\overline{\mathcal{O}_T(x)}=X\right\}=\bigcap_{n=1}^{+\infty}\bigcup_{m=0}^{+\infty}T^{-m}(U_n),\quad \textrm{where }\mathcal{O}_T(x)=\left\{T^nx\colon n\in\mathbb{Z}_+\right\},
\end{equation*}
is a dense $G_\delta$ set in $X$ because $\bigcup_{m=0}^{+\infty}T^{-m}(U_n)$ is open and dense in $X$.
It is easy to see that this standard argument for topological transitivity does not work here for Sigmund's conjecture (or Theorem~\ref{thm1.4}). So we need a new idea that will explore the times of sojourn of a motion in a domain.

Given any $x\in X$, let $\mathscr{U}_x$ be the neighborhood system of the point $x$ in $X$. To prove Conjecture~\ref{conj1.3}, we will need a classical result, which shows that $\mathcal{C}_x$ consists of the points $y\in X$ which are interesting for $x$.

\begin{lemma}[\cite{Hilmy,NS} for $C^0$-flow]\label{lem2.1}
Let $f\colon\mathbb{R}_+\times X\rightarrow X$ be a $C^0$-semi flow on a metric space $X$. Then for any $x\in X$, there holds
\begin{equation*}
\mathcal{C}_x=\left\{y\in X\,\bigg{|}\,\limsup_{T\to+\infty}\frac{1}{T}\int_0^T\mathds{1}_U(f(t,x))dt>0\ \forall U\in\mathscr{U}_y\right\}.
\end{equation*}
\end{lemma}

\begin{proof}
For self-closeness, we present an independent proof for this lemma which is shorter than that of \cite{NS}. Let $x\in X$ and write
$$
I(x)=\left\{y\in X\,\bigg{|}\,\limsup_{T\to+\infty}\frac{1}{T}\int_0^T\mathds{1}_U(f(t,x))dt>0\ \forall U\in\mathscr{U}_y\right\}.
$$
We first claim that $\mathcal{C}_x\subseteq I(x)$. Indeed, for any $q\in\mathcal{C}_x$, let $U\in\mathscr{U}_q$ be a neighborhood of $q$ in $X$; then
$$
\limsup_{T\to+\infty}\frac{1}{T}\int_0^T \mathds{1}_{U}(f(t,x))dt>0.
$$
Otherwise, one would find some $\varepsilon>0$ so that
$$
\lim_{T\to+\infty}\frac{1}{T}\int_0^T \mathds{1}_{B_{3\varepsilon}(q)}(f(t,x))dt=0.
$$
Further $\mathcal{C}_x-B_{2\varepsilon}(q)$ is a center of attraction of the motion $f(t,x)$ and we thus arrive at a contradiction to the minimality of $\mathcal{C}_x$.

Finally we assert that $\mathcal{C}_x\supseteq I(x)$. By contradiction, let $q\in I(x)-\mathcal{C}_x$ and then we can take some $\varepsilon>0$ such that $d(q,\mathcal{C}_x)\ge 3\varepsilon$. Set
\begin{gather*}
N(B_\varepsilon(q))=\{t\ge0\,|\, f(t,x)\in B_{\varepsilon}(q)\textrm{ for }1\le i\le l\}
\intertext{and}
N(B_\varepsilon(\mathcal{C}_x))=\{t\ge0\,|\, f(t,x)\in B_{\varepsilon}(\mathcal{C}_x)\textrm{ for }1\le i\le l\}.
\end{gather*}
Clearly $N(B_\varepsilon(q))\cap N(B_\varepsilon(\mathcal{C}_x))=\varnothing$. However, by definitions, $N(B_\varepsilon(q))$ has positive upper density and $N(B_\varepsilon(\mathcal{C}_x))$ has density $1$ in $[0,\infty)$. This is a contradiction.

The proof of Lemma~\ref{lem2.1} is thus completed.
\end{proof}

In Definition~\ref{def1.1} we do not require the $f$-invariance of $\mathcal{C}_x$. However, it is actually $f$-invariant by Lemma~\ref{lem2.1}.

\begin{cor}\label{cor2.2}
Let $f\colon\mathbb{R}_+\times X\rightarrow X$ be a $C^0$-semi flow on a metric space $X$. For any $x\in X$, $\mathcal{C}_x$ is $f$-invariant.
\end{cor}

Since for any real number $\theta>0$ and any integer $N\ge0$ there holds
\begin{equation*}
0\le\int_0^\theta\mathds{1}_U(f(N\theta+t,x))dt\le\theta,
\end{equation*}
then Lemma~\ref{lem2.1} implies immediately the following.

\begin{cor}\label{cor2.3}
Let $f\colon\mathbb{R}_+\times X\rightarrow X$ be a $C^0$-semi flow on a Polish space $X$. For any $x\in X$ and any $\theta>0$, there holds
\begin{equation*}
\mathcal{C}_x=\left\{y\in X\,\bigg{|}\,\limsup_{\mathbb{N}\ni N\to+\infty}\frac{1}{N}\int_0^{N\theta}\mathds{1}_U(f(t,x))dt>0\ \forall U\in\mathscr{U}_y\right\}.
\end{equation*}
\end{cor}

\noindent Here $\mathbb{N}=\{1,2,\dotsc\}$.
\begin{proof}
The statement follows from that $T=N_T\theta+r_T$ where $N_T\in\mathbb{N}, 0\le r_T<\theta$ for any $T\ge1$ and that
\begin{equation*}
\limsup_{T\to+\infty}\frac{1}{T}\int_0^{T}\mathds{1}_U(f(t,x))dt=\limsup_{\mathbb{N}\ni N\to+\infty}\frac{1}{N\theta}\int_0^{N\theta}\mathds{1}_U(f(t,x))dt
\end{equation*}
for any $U\in\mathscr{U}_y$ and any $y\in X$.
\end{proof}

We are now ready to prove one of our main statements---Theorem~\ref{thm1.4}---by applying Lemma~\ref{lem2.1} and Corollaries~\ref{cor2.2} and \ref{cor2.3}.

\begin{proof}[Proof of Theorem~\ref{thm1.4}]
Write $\Theta=\{x\in X\colon \mathcal{C}_x=X\}$ and let $\x\in X$ be such that $\mathcal{C}_{\x}=X$. Then from Lemma~\ref{lem2.1}, it follows that the orbit $\mathcal{O}_f(\x)=f(\mathbb{R}_+,\x)$ is dense in $X$ and that $\mathcal{O}_f(\x)\subseteq\Theta$ such that for any $y\in\mathcal{O}_f(\x)$,
\begin{equation*}
\limsup_{\mathbb{N}\ni N\to\infty}\frac{1}{N}\int_0^N\mathds{1}_U(f(t,\x))dt=\limsup_{\mathbb{N}\ni N\to\infty}\frac{1}{N}\int_0^N\mathds{1}_U(f(t,y))dt,
\end{equation*}
for every nonempty $U\in\mathscr{T}_X$, where $\mathscr{T}_X$ is the topology of the space $X$.

Let $\mathscr{U}=\{U_i\}_{i=1}^\infty$ be an any given countable base of the topology $\mathscr{T}_X$ of the state space $X$. Then by Corollary~\ref{cor2.3}, we can choose a sequence of positive integers $L_i, i=1,2,\dotsc$, such that
\begin{equation*}
\limsup_{\mathbb{N}\ni N\to\infty}\frac{1}{N}\int_0^N\mathds{1}_{U_i}(f(t,\x))dt>\frac{1}{L_i},
\end{equation*}
for all $i=1,2,\dotsc$. For each $i$, write
\begin{equation*}
\Theta_i=\left\{y\in X\,\big{|}\,\forall n_0\in\mathbb{N}, \exists n>n_0\textrm{ with }\int_0^{nL_i}\mathds{1}_{U_i}(f(t,y))dt>n\right\}.
\end{equation*}
From the continuity of $f(t,y)$ with respect to $(t,y)$ and
\begin{equation*}
\Theta_i=\bigcap_{n_0=1}^\infty\bigcup_{n>n_0}\left\{y\in X\,\big{|}\,\int_0^{nL_i}\mathds{1}_{U_i}(f(t,y))dt>n\right\},
\end{equation*}
it follows that $\Theta_i$ is a $G_\delta$ subset of $X$. Because $\x$ belongs to $\Theta_i$ by noting $N=n_NL_i+r_N$ and $\lim_{N\to\infty}\frac{r_N}{N}=0$ where $0\le r_N<L_i$ and then similarly $\mathcal{O}_f(\x)\subseteq\Theta_i$, there follows
$\bigcap_{i=1}^\infty\Theta_i$ is a dense and $G_\delta$ set in $X$.

Since for any $y\in\bigcap_{i=1}^\infty\Theta_i$ we have
\begin{equation*}
\limsup_{\mathbb{N}\ni N\to\infty}\frac{1}{N}\int_0^N\mathds{1}_{U_i}(f(t,y))dt\ge\frac{1}{L_i}>0
\end{equation*}
for any $U_i\in\mathscr{U}$ and $\mathscr{U}$ is a base of the topology $\mathscr{T}_X$ of the space $X$, we see that $y\in\Theta$ from Corollary~\ref{cor2.3}.
Therefore, $\Theta$ is a residual set in $X$.

This completes the proof of Theorem~\ref{thm1.4}.
\end{proof}

Finally we mention that the statement of Theorem~\ref{thm1.4} is also valid for a continuous transformation $f\colon X\rightarrow X$ of a Polish space $X$.

\section{Li-Yorke chaotic pairs and sensitive dependence on initial data}\label{sec3}
In this section, we shall apply Theorem~\ref{thm1.4} to the study of chaos of a topological dynamical system on a Polish space $X$.

Let $f\colon\mathbb{R}_+\times X\rightarrow X$ be a $C^0$-semi flow on the Polish space $(X,d)$. Recall that two points $x,y\in X$ is called a \textit{Li-Yorke chaotic pair} for $f$ if
\begin{gather*}
\limsup_{t\to+\infty}d(f(t,x),f(t,y))>0\quad \textrm{and}\quad\liminf_{t\to+\infty}d(f(t,x),f(t,y))=0.
\end{gather*}
That is to say, $x$ is proximal to $y$ but not asymptotical. If there can be found an uncountable set $S\subset X$ such that every pair of points $x,y\in S, x\not=y$, is a Li-Yorke chaotic pair for $f$, then we say $f$ is \textit{Li-Yorke chaotic}; see, e.g., Li and Yorke 1975~\cite{LY}.

\subsection{Nongeneric case}\label{sec3.1}
Let $f\colon\mathbb{R}_+\times X\rightarrow X$ be a $C^0$-semi flow on a compact metric space. The following theorem shows that if $\mathcal{C}_x$ has no the generic dynamics in the sense of Definition~\ref{def1.5}, then $f$ has the chaotic dynamics.

\begin{thm}\label{thm3.1}
Given any $x\in X$, if $\mathcal{C}_x$ is not generic, then one can find some point $q\in\Delta$, for any closed $f$-invariant subset $\Delta\subseteq\mathcal{C}_x$, such that $(x,q)$ is a Li-Yorke chaotic pair for $f$.
\end{thm}

\begin{proof}
Given any $x\in X$, let $\mathcal{C}_x$ be not generic in the sense of Definition~\ref{def1.5}. Let $\Delta$ be an $f$-invariant nonempty closed subset of $\mathcal{C}_x$. Then by Theorem~\ref{thm1.4} it follows that $x\not\in\mathcal{C}_x$. Moreover from Definition~\ref{def1.1}, we can obtain that $x$ is proximal to $\Delta$; that is, $$\liminf_{t\to+\infty}d(f(t,x), f(t,\Delta))=0.$$
Then from \cite{Aus, Ell} also \cite[Proposition~8.6]{Fur}, it follows that there exists some point $q\in\Delta$ such that
$$
\liminf_{t\to+\infty}d(f(t,x),f(t,q))=0.
$$
We claim that $\limsup_{t\to+\infty}d(f(t,x),f(t,q))>0$. Otherwise, $\lim_{t\to+\infty}d(f(t,x),f(t,q))=0$ and thus $\mathcal{C}_x=\mathcal{C}_q$; and then $\mathcal{C}_x$ is generic by Theorem~\ref{thm1.4}.

The proof of Theorem~\ref{thm3.1} is therefore complete.
\end{proof}

\begin{cor1}
Given any $x\in X$, if $\mathcal{C}_x$ is not generic, then one can find some point $q\in\mathcal{C}_x$ such that $(x,q)$ is a Li-Yorke chaotic pair for $f$ and the set
\begin{gather*}
N_f(x,B_\varepsilon(q))=\{t\ge0\colon f(t,x)\in B_\varepsilon(q)\}
\end{gather*}
is a central set in $\mathbb{R}_+$, which has positive upper density.
\end{cor1}

\begin{proof}
Let $\Delta\subset\mathcal{C}_x$ be a minimal set. Then there can be found by Theorem~\ref{thm3.1} a point $q\in\Delta$ such that $(x,q)$ is a Li-Yorke chaotic pair for $f$. By definition (cf.~\cite[Definition~8.3]{Fur} and \cite[Definition~7.2]{Dai-p}) $N_f(x,B_\varepsilon(q))$ is a central set of $\mathbb{R}_+$ for each $\varepsilon>0$.
In addition, by Lemma~\ref{lem2.1} it follows that $N_f(x,B_\varepsilon(q))$ has positive upper density. This proves the corollary.
\end{proof}

\begin{cor2}
Let there exist a fixed point or a periodic orbit in the minimal center $\mathcal{C}_x$ of attraction of a motion $f(t,x)$. Then we can find a Li-Yorke chaotic pair near $\mathcal{C}_x$.
\end{cor2}

\begin{proof}
First if $\mathcal{C}_x$ is generic in the sense of Definition~\ref{def1.5}, then $f$ restricted to it is topologically transitive and has a fixed point or a periodic orbit. Then by Huang and Ye 2002~\cite[Theorem~4.1]{HY}, it follows that $f$ restricted to $\mathcal{C}_x$ is chaotic in the sense of Li and Yorke.

On the other hand, if $\mathcal{C}_x$ is not generic in the sense of Definition~\ref{def1.5}, then from Theorem~\ref{thm3.1} it follows that there always exists a Li-Yorke chaotic pair near $\mathcal{C}_x$.

The proof of the corollary is thus complete.
\end{proof}

Recall that a motion $f(t,x)$ is referred to as a \textit{Birkhoff recurrent motion} of $f$ if $\overline{\mathcal{O}_f(x)}$ is minimal (cf.~\cite{NS,CD}). It is also called ``uniformly recurrent'' in \cite{Fur} and ``almost periodic'' of von Neumann in \cite{GH} in the discrete-time case.

Motivated by \cite{BB, GW, Dai} we can obtain the following theorem on sensitivity on initial data near the minimal center of attraction of a motion.

\begin{thm}\label{thm3.2}
Let $\mathcal{C}_x$, for a motion $f(t,x)$, be not generic. If the Birkhoff recurrent points of $f$ are dense in $\mathcal{C}_x$, then $f$ has the sensitive dependence on initial data near $\mathcal{C}_x$ in the sense that one can find a sensitive constant $\epsilon>0$ such that for any $a\in X, \hat{x}\in\mathcal{C}_x$ and any $U\in\mathscr{U}_{\hat{x}}$, there exists some point $c\in U$ with $\limsup_{t\to+\infty}d\big{(}f(t,a), f(t,c)\big{)}\ge\epsilon$.
\end{thm}

\begin{proof}
Since $\mathcal{C}_x$ is not generic, by Theorem~\ref{thm1.4} it follows that it is not minimal and so it contains at least two different motions of $f$ far away each other. Thus one can find a number $\delta>0$ such that for all $\hat{x}\in \mathcal{C}_x$ there exists a corresponding motion $f(t,q)$ in $\mathcal{C}_x$, not necessarily recurrent, such that
\begin{equation*}
d\big{(}\hat{x},\overline{\mathcal{O}_f(q)}\big{)}\ge\delta,
\end{equation*}
where $d(\hat{x},A)=\inf_{a\in A}d(\hat{x},a)$ for any subset $A$ of $X$. We will show that $f$ has sensitive dependence on initial data with sensitivity constant $\epsilon=\delta/4$ following the idea of \cite[Theorem~4]{Dai}.

For this, we let $\hat{x}$ be an arbitrary point in $\mathcal{C}_x$ and let $U$ be an arbitrary neighborhood of $\hat{x}$ in $X$.
Since the Birkhoff recurrent motions of $(X,f)$ are dense in $\mathcal{C}_x$ from assumption of the theorem, there exists a Birkhoff recurrent point $p\in U\cap B_{\epsilon/2}(\hat{x})\cap\mathcal{C}_x$, where $B_r(\hat{x})$ is the open ball of radius $r$ centered at $\hat{x}$ in $X$. As we noted above, there must exist another point $q\in \mathcal{C}_x$ whose orbit $\mathcal{O}_f(q)$ is of distance at least $4\epsilon$ from the given point $\hat{x}$.

Let $\eta>0$ be such that $\eta<\epsilon/2$. Then from the Birkhoff recurrence of the motion $f(t,p)$, it follows that one can find a constant $T=T(\eta,p)>0$ such that for any $\gamma\ge0$, there is some moment $t_\gamma\in[\gamma,\gamma+T)$ verifying that
\begin{equation*}
d\big{(}p,f^{t_\gamma}(p)\big{)}<\eta.
\end{equation*}
For the given $q$, we simply write
\begin{equation*}
V=\bigcap_{t\in[0,2T)}f^{-t}(B_\epsilon(f^t(q))),\quad \textrm{where }f^{-t}(\cdot)=f(t,\cdot)^{-1}.
\end{equation*}
Clearly from the continuity of topological flow, it follows that $V$ is a neighborhood of $q$ in $X$ but not necessarily open, and it is nonempty since $q\in V$.

Since $\mathcal{C}_x$ is the minimal center of attraction of the motion $f(t,x)$, from Lemma~\ref{lem2.1} it follows that there is at least one point $z\in U\cap B_\epsilon(\hat{x})$ such that $f^N(z)\in V$ for some sufficiently large number $N\gg T$. Let
$$
N=jT-r\quad \textrm{where }0\le r<T,\ j\in\mathbb{N},
$$
and
$$
t_{jT}\in[jT,(j+1)T)\quad \textrm{such that }d(p,f^{t_{jT}}(p))<\eta.
$$
Then $0\le t_{jT}-N<2T$.

By construction, one has
$$
f^{t_{jT}}(z)=f^{t_{jT}-N}(f^N(z))\in f^{t_{jT}-N}(V)\subseteq B_\epsilon(f^{t_{jT}-N}(q)).
$$
From the triangle inequality of metric, it follows that
\begin{equation*}\begin{split}
d(f^{t_{jT}}(p),f^{t_{jT}}(z))&\ge d(p,f^{t_{jT}}(z))-\eta\\
&\ge d(\hat{x},f^{t_{jT}}(z))-d(p,\hat{x})-\eta\\
&\ge d(\hat{x},f^{t_{jT}-N}(q))-d(f^{t_{jT}-N}(q),f^{t_{jT}}(z))-d(p,\hat{x})-\eta.
\end{split}\end{equation*}
Consequently, since $\eta<\epsilon/2$, $p\in B_{\epsilon/2}(\hat{x})$ and $f^{t_{jT}}(z)\in B_\epsilon(f^{t_{jT}-N}(q))$, it holds that
$$
d(f^{t_{jT}}(p),f^{t_{jT}}(z))>2\epsilon.
$$
Therefore from the triangle inequality again, one can obtain either
$$
d(f^{t_{jT}}(\hat{x}),f^{t_{jT}}(z))>\epsilon
$$
or
$$
d(f^{t_{jT}}(\hat{x}),f^{t_{jT}}(p))>\epsilon.
$$
Repeating this argument for another likewise $N$ bigger than $(j+2)T$, one can find a sequence $t_n=j_nT\uparrow+\infty$ as $n\to+\infty$ such that
either
$$d(f^{t_n}(\hat{x}),f^{t_n}(z))>\epsilon$$
or
$$d(f^{t_n}(\hat{x}),f^{t_n}(p))>\epsilon,$$
for all $n\ge1$. Thus in either case, we have found a point $\hat{y}\in U$ such that
$$
\limsup_{t\to+\infty}d(f^t(\hat{x}),f^t(\hat{y}))\ge\epsilon.
$$
Now for any $a\in X$, using the triangle inequality once more, we see either
\begin{gather*}
\limsup_{t\to+\infty}d(f^t(\hat{x}),f^t(a))\ge\frac{\epsilon}{3}\intertext{or}\limsup_{t\to+\infty}d(f^t(\hat{y}),f^t(a))\ge\frac{\epsilon}{3}.
\end{gather*}
Since $\hat{x}, U$ both are arbitrary and $\hat{y}\in U$, hence the proof of Theorem~\ref{thm3.2} is complete.
\end{proof}

We note that if $(\mathcal{C}_x,f)$ is distal (cf.~\cite[Definition~8.2]{Fur}) and not minimal and even not topologically transitive, then the conditions of Theorem~\ref{thm3.2} hold; i.e., the Birkhoff recurrent points are dense in $\mathcal{C}_x$. In fact, $\mathcal{C}_x$ consists of Birkhoff recurrent points (\cite[Corollary of Theorem~8.7]{Fur}) and it is not topologically transitive.

\subsection{Generic case}\label{sec3.2}
Let $f\colon\mathbb{R}_+\times X\rightarrow X$ be a $C^0$-semi flow on a Polish space and $f(t,p)$ a Lagrange stable motion. Then the minimal center $\mathcal{C}_p$ of attraction of the motion $f(t,p)$ is always existent. Theorem~\ref{thm1.7} is just a corollary of the following

\begin{thm}\label{thm3.3}
Let $\mathcal{C}_p$ be generic and not a minimal subset of $(X,f)$. Then there exists a residual subset $S$ of $\mathcal{C}_p$ such that for any $x\in S$ and any minimal subset $\Lambda\subset\mathcal{C}_p$, there corresponds some point $y\in\Lambda$ with the properties: $x,y$ form a Li-Yorke chaotic pair for $f$ and
\begin{equation*}
\liminf_{t\to+\infty}d(f(t,x),y)=0\quad \textrm{and}\quad\limsup_{t\to+\infty}d(f(t,x),y)\ge\frac{1}{2}\mathrm{diam}(\mathcal{C}_p).
\end{equation*}
\end{thm}

\begin{proof}
Since $\mathcal{C}_p$ is generic in the sense of Definition~\ref{def1.5}, there exists some point $q\in\mathcal{C}_p$ with $\mathcal{C}_q=\mathcal{C}_p$. Therefore by Theorem~\ref{thm1.4}, there is a residual subset $S$ of $\mathcal{C}_p$ such that $\mathcal{C}_x=\mathcal{C}_p$ for all point $x$ in $S$.
Because $\mathcal{C}_p$ is not a minimal subset of $X$ by hypothesis of Theorem~\ref{thm1.7}, $\overline{\mathcal{O}_f(x)}$ is not minimal for each $x\in S$. 

Let $\Lambda$ be a minimal subset of $\mathcal{C}_p$. Then each $x\in S$ is proximal to $\Lambda$.
Moreover by \cite[Theorem~8.7]{Fur}, it follows that for every $x\in S$, there corresponds some point $y\in\Lambda$ such that $x$ is proximal to $y$ and $f(t,y)$ is Birkhoff recurrent (or uniformly recurrent).
Clearly, $x$ and $y$ is a Li-Yorke chaotic pair for $f$. In addition, from Lemma~\ref{lem2.1} follows that 
\begin{gather*}
\limsup_{t\to+\infty}d(f(t,x),y)\ge\frac{1}{2}\mathrm{diam}(\mathcal{C}_p)\\
\intertext{and} 
\liminf_{t\to+\infty}d(f(t,x),y)=0.
\end{gather*}
This completes the proof of Theorem~\ref{thm3.3}.
\end{proof} 

Therefore by Theorems~\ref{thm3.1} and \ref{thm3.3}, it follows that every Lagrange stable motion $f(t,x)$ is chaotic in the sense of Definition~\ref{def1.6} if its minimal center $\mathcal{C}_x$ of attraction is not a minimal subset of $(X,f)$.
\section{Quasi-weakly almost periodic motion}\label{sec4}
In this section, we let $f\colon\mathbb{R}_+\times X\rightarrow X$ be a $C^0$-semi flow on the compact metric space $X$.

\begin{defn}[{Huang-Zhou 2012~\cite{HZ}}]\label{def4.1}
A point $x$ in $X$ is called a \textit{quasi-weakly almost periodic point} of $f$ if for any $\varepsilon>0$ there exists an integer $N=N(\varepsilon)\ge1$ and an increasing positive integer sequence $\{n_j\}$ with the property that for each $j$ there are $0=t_0<t_1<\dotsm<t_{n_j}<n_jN$ with $t_{i+1}-t_i\ge1$ such that $f(t_i,x)\in B_\varepsilon(x)$ for all $i=1,\dotsc,n_j$.
\end{defn}

As results of the statements of Lemma~\ref{lem2.1} and Theorem~\ref{thm1.4}, we can obtain the following two results.

\begin{prop}\label{prop4.2}
The following statements are equivalent to each other.
\begin{enumerate}
\item[$(1)$] $x\in X$ is a quasi-weakly almost periodic point of $f$.
\item[$(2)$] $x\in\mathcal{C}_x$.
\end{enumerate}
\end{prop}

\begin{proof}
(1)$\Rightarrow$(2) follows from Lemma~\ref{lem2.1}. (2)$\Rightarrow$(1) follows from Lemma~\ref{lem2.1} and the local section theorem of Bebutov (cf.~\cite[Theorem~V.2.14]{NS}).
\end{proof}

Proposition~\ref{prop4.2} has been proved in \cite{HZ} in the case where $x$ is a Poisson stable point of $f$, i.e., there is a sequence $t_n\to\infty$ such that $f(t_n,x)\to x$ as $n\to\infty$.

\begin{prop}
If $x\in\mathcal{C}_x$, then the set $\{y\in\mathcal{C}_x\,|\, y\in\mathcal{C}_y=\mathcal{C}_x\}$ is dense and $G_\delta$ relative to the subspace $\mathcal{C}_x$.
\end{prop}

\begin{proof}
The statement follows from Theorem~\ref{thm1.4} with $\mathcal{C}_x$ replacing of $X$.
\end{proof}

\section{Minimal center of multi-attraction of a motion}\label{sec5}
Let $f\colon\mathbb{R}_+\times X\rightarrow X$ be a $C^0$-semi flow on a Polish space $X$.
From now on, by $\lambda(dt)$ we denote the standard Haar (Lebesque) measure on $\mathbb{R}$. We will need the following simple but useful fact.

\begin{lemma}\label{lem5.1}
Let $S$ be a measurable subset of $\mathbb{R}_+$ and $\tau>0$. If $S$ has the density $\alpha$, i.e.,
\begin{gather*}
D(S):=\lim_{T\to+\infty}\frac{\lambda(S\cap[0,T])}{T}=\alpha,
\end{gather*}
then $\tau S=\{\tau t\colon t\in S\}$ also has the density $\alpha$ in $\mathbb{R}_+$.
\end{lemma}

\begin{proof}
Let $\tau>0$ be any given. Since $\lambda(\tau A)=\tau\lambda(A)$ and $\lambda(\tau S\cap[0,T])=\tau\lambda(S\cap[0,T\tau^{-1}])$, hence it follows that $D(\tau S)=1$. This proves the lemma.
\end{proof}

It should be noted here that there is no an analogous result for the discrete-time $\mathbb{Z}_+$; for example, $S=\{0,2,4,6,\dotsc\}$ has the density $\frac{1}{2}$ but $\frac{1}{2}S$ has the density $1$ in $\mathbb{Z}_+$.

The following lemma shows that every minimal center of attraction of a motion $f(t,x)$ is multiply attracting as $t\to+\infty$.

\begin{lemma}\label{lem5.2}
Let $\mathcal{C}_x$ be existent for a motion $f(t,x)$ as $t\to+\infty$. Then for any $t_1>0, \dotsc, t_l>0$ where $l\in\mathbb{N}$ and any $\varepsilon>0$,
\begin{gather*}
\lim_{T\to+\infty}\frac{1}{T}\int_0^T\mathds{1}_{B_\varepsilon(\mathcal{\mathcal{C}}_x)}(f(t_1t,x))\dotsm\mathds{1}_{B_\varepsilon(\mathcal{\mathcal{C}}_x)}(f(t_lt,x))dt=1.
\end{gather*}
\end{lemma}

\begin{proof}
For any $\tau>0$ and any open set $U$, define an open set
$$
N_\tau(x,U)=\{t\colon 0\le t<+\infty, f(\tau t,x)\in U\}.
$$
It is easy to check that $\tau^{-1}N_1(x,U)=N_\tau(x,U)$.
Then by Lemma~\ref{lem5.1} and Definition~\ref{def1.1}, it follows that $N_{t_1}(x,B_\varepsilon(\mathcal{C}_x)), \dotsc, N_{t_l}(x,B_\varepsilon(\mathcal{C}_x))$ all have the density $1$. Thus
$$
N_{t_1,\dotsc,t_l}(x,B_\varepsilon(\mathcal{C}_x)):=N_{t_1}(x,B_\varepsilon(\mathcal{C}_x))\cap\dotsm\cap N_{t_l}(x,B_\varepsilon(\mathcal{C}_x))
$$
also has the density $1$ in $\mathbb{R}_+$. This completes the proof of Lemma~\ref{lem5.2}.
\end{proof}

Recall that a motion $f(t,x)$ is said to be \textit{Lagrange stable} as $t\to+\infty$ if the orbit-closure $\overline{\mathcal{O}_f(x)}$ is compact in $X$ (cf.~\cite{NS}). As a direct result of Lemma~\ref{lem5.2}, we can obtain the following.

\begin{cor}\label{cor5.3}
For any Lagrange stable motion $f(t,x)$ as $t\to+\infty$, $\mathcal{C}_x$ is the minimal closed subset of $X$ such that for any $t_1>0, \dotsc, t_l>0$ and any $\varepsilon>0$,
\begin{gather*}
\lim_{T\to+\infty}\frac{1}{T}\int_0^T\mathds{1}_{B_\varepsilon(\mathcal{\mathcal{C}}_x)}(f(t_1t,x))\dotsm\mathds{1}_{B_\varepsilon(\mathcal{\mathcal{C}}_x)}(f(t_lt,x))dt=1.
\end{gather*}
\end{cor}

This result shows that $\mathcal{C}_x$ is the ``minimal center of multi-attraction'' of a Lagrange stable motion $f(t,x)$ as $t\to+\infty$.

\section*{Acknowledgment}
This work was supported partly by National Natural Science Foundation of China grant $\#$11431012, 11271183 and PAPD of Jiangsu Higher Education Institutions.


\end{document}